\documentclass[11pt]{amsart}

\usepackage[latin1]{inputenc}
\usepackage{amssymb,amsmath,amsfonts,bbm}
\usepackage{enumerate}

\usepackage[a4paper,twoside,left=2.8cm,right=2.5cm,top=3.1cm,bottom=2.3cm]{geometry}

\newtheorem{Proposition}{Proposition}[section]

\newtheorem{Lemma}[Proposition]{Lemma}
\newtheorem{Theorem}[Proposition]{Theorem}

\DeclareMathOperator{\vol}{Vol}

\DeclareMathOperator{\Val}{Val}

\DeclareMathOperator{\SL}{SL}
\DeclareMathOperator{\GL}{GL}

\DeclareMathOperator{\D}{D}
\DeclareMathOperator{\sgn}{sgn}

\newcommand{\R}{\mathbb{R}}
\newcommand{\C}{\mathbb{C}}

\newcommand{\K}{\mathcal{K}}

\newcommand{\func}[5]{\ensuremath{\begin{array}{cccl}
#1:&#2&\longrightarrow&#3\\&#4&\mapsto&#5\end{array}}}

\title[Minkowski valuations in a 2-dimensional complex vector space]{Minkowski valuations in a 2-dimensional complex vector space}
\author{Judit Abardia} 
\address{Institut f\"ur Mathematik, Goethe-Universit\"at Frankfurt, 
Robert-Mayer-Str. 10, 60054 Frankfurt, Germany}

\email{abardia@math.uni-frankfurt.de}

\begin{document}

\begin{abstract}The classification of continuous, translation invariant Minkowski valuations which are contravariant (or covariant) with respect to the complex special linear group is established in a 2-dimensional complex vector space. Every such valuation is given by the sum of a valuation of degree of homogeneity 1 and 3. In dimensions $m\geq 3$ such a classification was previously established and only valuations of a degree of homogeneity $2m-1$ appear.
\end{abstract}
\thanks{Supported by DFG grant BE 2484/3-1}
\date{\today}
\subjclass[2000]{52B45, 
52A39
}

\maketitle 

\section{Introduction}
Let $V$ denote a real vector space of dimension $n$ and $\mathcal{K}(V)$ the space of compact convex bodies in $V$, endowed with the Hausdorff topology. An operator $Z:\mathcal{K}(V)\to(A,+)$ with $(A,+)$ an abelian semi-group is called a {\em valuation} if it satisfies the following additivity property $$Z(K\cup L)+ Z(K\cap L)=Z(K)+Z(L),$$ for all $K,L\in\mathcal{K}(V)$ such that $K\cup L\in\mathcal{K}(V)$. If $(A,+)$ is the set of convex bodies endowed with the Minkowski addition, then $Z$ is called a {\em Minkowski valuation}. This class of valuations has been widely studied, see for instance \cite{haberl10,kiderlen,ludwig02, ludwig_2005,ludwig06_survey, ludwig10,parapatits.schuster12,schneider_schuster06, schuster.conv,schuster10,schuster.wannerer12,wannerer11}.

A Minkowski valuation $Z:\mathcal{K}(V)\to\mathcal{K}(V)$ is called {\em $\mathrm{SL}(V,\mathbb{R})$-covariant} if $$Z(gK)=gZ(K),\quad\forall g\in\mathrm{SL}(V,\mathbb{R}),$$
where $\SL(V,\R)$ denotes the special linear group. 
A Minkowski valuation $Z:\mathcal{K}(V)\to\mathcal{K}(V^*)$ is {\em $\mathrm{SL}(V,\mathbb{R})$-contravariant} if $$Z(gK)=g^{-*}Z(K),\quad\forall g\in\mathrm{SL}(V,\mathbb{R}),$$
where $V^*$ denotes the dual space of $V$ and $g^{-*}$ denotes the inverse of the adjoint map of $g$. 

Two well-known examples of Minkowski valuations are the projection body and difference body operators. 
The \emph{projection body of $K\in\K(V)$} is the convex body $\Pi K\in\K(V^*)$ with support function 
$$h(\Pi K,v)=\frac{n}{2}V(K,\dots,K,[-v,v]),\quad\forall v\in V,$$
where $V(K,\dots,K,[-v,v])$ denotes the mixed volume of $n-1$ copies of $K$ and one copy of the segment joining $-v$ and $v$.
The operator $\Pi$ constitutes an example of a continuous, translation invariant Minkowski valuation which is $\mathrm{SL}(V,\mathbb{R})$-contravariant (see \cite{petty67}). Ludwig proved in \cite{ludwig02,ludwig_2005} that in a real vector space of dimension $n\geq 2$ the projection body operator is the only (up to a positive constant) continuous, translation invariant and $\mathrm{SL}(V,\mathbb{R})$-contravariant Minkowski valuation. 

For the covariant case, it follows from the work of Ludwig \cite{ludwig_2005} that the difference body is the unique (up to a positive constant) continuous Minkowski valuation which is translation invariant and $\mathrm{SL}(V,\mathbb{R})$-covariant. The {\em difference body of a convex body $K\in\mathcal{K}(V)$} is defined by $$\mathrm{D}\!K=K+(-K),$$ where $-K$ denotes the reflection of $K$ about the origin. 

In \cite{abardia12, abardia.bernig}, the complex analog of the previous results was studied. More precisely, in a complex vector space of complex dimension $m\geq 3$ a classification result for those Minkowski valuations which are continuous, translation invariant and $\mathrm{SL}(W,\mathbb{C})$-contravariant or $\mathrm{SL}(W,\mathbb{C})$-covariant was given. In this framework, other valuations than the ones appearing in the real case have to be considered.  
Related results concerning convex bodies or valuations in a complex vector space as ambient space can be found in \cite{alesker03_un,bernig_fu_hig,fu06,koldobsky_koenig_zymonopoulou08,koldobsky.paouris.zymonopoulu,rubin10}.

The classification result for the $\mathrm{SL}(W,\mathbb{C})$-contravariant valuations states the following. 

\begin{Theorem}[\cite{abardia.bernig}]\label{contram}Let $W$ be a complex vector space of complex dimension $m \geq 3$. A map $Z:\mathcal{K}(W) \to \mathcal{K}(W^*)$ is a continuous, translation invariant and $\mathrm{SL}(W,\mathbb{C})$-contravariant Minkowski valuation if and only if there exists a convex body $N\subset\mathbb{C}$ such that $Z=\Pi_N$, where
$\Pi_N K \in \mathcal{K}(W^*)$ is the convex body with support function 
\begin{equation}\label{eqcontram}h(\Pi_{N}K,u)=V(K,\dots,K,N\cdot u),\quad\forall u \in W,\end{equation}
with $N\cdot u=\{cu\,:\,c\in N\subset \C\}$. Moreover, $N$ is unique up to translations. 
\end{Theorem}

The result in the $\mathrm{SL}(W,\mathbb{C})$-covariant case reads as follows.
\begin{Theorem}[\cite{abardia12}]\label{covm}Let $W$ be a complex vector space of complex dimension $m \geq 3$. A map $Z:\mathcal{K}(W) \to \mathcal{K}(W)$ is a continuous, translation invariant and $\mathrm{SL}(W,\mathbb{C})$-covariant Minkowski valuation if and only if there exists a convex body $M\subset\mathbb{C}$ such that $Z=\mathrm{D}_M$, where
$\mathrm{D}_M K \in \mathcal{K}(W)$ is the convex body with support function 
\begin{equation}\label{eqcovm}h(\mathrm{D}_{M} K,\xi)=\int_{S^1}h(\alpha K,\xi)dS(M,\alpha),\quad\forall \xi \in W^*.\end{equation}
Here $dS(M,\cdot)$ denotes the area measure of $M$, and $\alpha K=\{\alpha  k\,:\, k\in K\subset W\}$ with $\alpha\in S^1\subset\mathbb{C}$. Moreover, $M$ is unique up to translations. 
\end{Theorem}

The necessity of the assumption $m\geq 3$ in Theorem \ref{contram} was already shown in \cite{abardia.bernig}, where a family of $\SL(W,\C)$-contravariant Minkowski valuations of homogeneity degree 1 was explicitly constructed when $m=2$, leaving the complete classification in the 2-dimensional case open until now. 
The fact that more operators appear in this situation is due to the existence of the following canonical identification between a 2-dimensional complex vector space $W$ and its dual space $W^*$:

Fix a basis of the 2-dimensional complex vector space $W$ and consider the determinant map 
\begin{equation}\label{detId}\func{\det}{W\times W}{\C}{(u,v)}{\det(u,v).}
\end{equation}
This map induces an identification $\Phi$ between $W$ and its dual space $W^*$, which satisfies $\Phi(gu)=(\det g) g^{-*}\Phi(u)$, for every $g\in\GL(W,\C)$, $u\in W$.

Thus, every $\SL(W,\C)$-contravariant (resp.\! covariant) Minkowski valuation $Z$ of degree $k$ induces an $\SL(W,\C)$-covariant (resp.\! contravariant) Minkowski valuation $\Phi^{-1}\circ Z$ (resp.\! $\Phi\circ Z$) also of degree $k$. 

In this note we prove the theorem below which gives a complete classification of the $\SL(W,\C)$-contravariant and $\SL(W,\C)$-covariant continuous, translation invariant Minkowski valuations in a 2-dimensional complex vector space. 

\begin{Theorem}\label{contravariant}Let $W$ be a 2-dimensional complex vector space. A map $Z:\K(W)\to\K(W^*)$ is a continuous, translation invariant and $\SL(W,\C)$-contravariant Minkowski valuation if and only if there are convex bodies $M,N\subset\C$ for which $ZK=\tilde\D_{M}K+\Pi_{N}K$, where $\tilde \D_{M}:=\Phi\circ\D_{M}:\K(W)\to\K(W^*)$ is defined by
\begin{equation}\label{contra1}h(\tilde\D_{M}K,w)=\int_{S^1}h(\det(K,w),\alpha)dS(M,\alpha),\quad K\in\K(W),\,w\in W,\end{equation}
with $\det(K,w):=\{\det(k,w)\,|\,k\in K\}$,
and $\Pi_{N}:\K(W)\to\K(W^*)$ is defined by 
\begin{equation}\label{contra3}h(\Pi_{N}K,w)=V(K,K,K,N\cdot w).
\end{equation} 
Moreover, $M$ and $N$ are unique up to translations. 
\end{Theorem}

Using the identification map $\Phi$ given by \eqref{detId}, the previous theorem also yields a classification of the Minkowski valuations $Z':\K(W)\to\K(W)$ which are continous, translation invariant and $\SL(W,\C)$-covariant. In this case, we have $Z'=\Phi^{-1}\circ Z$, that is, $Z'K=\D_{M}K+(\Phi^{-1}\circ\Pi_{N})K$ for some convex bodies $M,N\in\K(\C)$. 

\subsection*{Acknowledgments}
I would like to thank Andreas Bernig, Monika Ludwig, Franz Schuster and Thomas Wannerer for interesting discussions and useful remarks on this paper.

\section{Background material}
For more information on the results stated in this section, we refer to the books \cite{gardner_book06,klain_rota, schneider_book93}.

\subsection{Support function}
Let $K \in \mathcal{K}(V)$. The {\em support function of $K$} is given by
\begin{align*} 
 h(K,\cdot):V^* & \to \R, \\
\xi & \mapsto \sup_{x \in K}\langle \xi,x\rangle,
\end{align*}
where $\langle \xi,x\rangle$ denotes the pairing of $\xi \in V^*$ and $x\in V$.

The support function is $1$-homogeneous (i.e. $h(K,t\xi)=th(K,\xi)$ for all $t \geq 0$) and subadditive (i.e. $h(K,\xi+\eta) \leq h(K,\xi)+h(K,\eta)$ for all $\xi,\eta\in V^*$). Moreover, if a function on $V^*$ is $1$-homogeneous and subadditive, then it is the support function of a unique compact convex set $K \in \K(V)$ (cf. \cite[Theorem 1.7.1]{schneider_book93}). Note that if $h(K,\xi)=h(K,-\xi)=0$ for some $\xi\in V^*$, then $K\subset \ker\xi\subset V$.

The support function is also linear with respect to the Minkowski sum on $\K(V)$ and has the following property 
$$h(gK,\xi)=h(K,g^*\xi),\quad \forall \xi \in V^*,\, g \in \GL(V,\R).$$
In a complex vector space $W$ this equality holds in particular for $g\in\GL(W,\C)$. 

\subsection{Mixed volumes}
In an $n$-dimensional real vector space $V$, the {\em mixed volume} is the unique symmetric and Minkowski multilinear map $(K_{1},\dots,K_{n})\mapsto V(K_{1},\dots,K_{n})$ on $n$-tuples of convex bodies with $V(K,\ldots,K)=\vol(K)$.

It is nonnegative, continuous and translation invariant in each
component. Moreover,  
$$V(gK_1,\dots,gK_n)=|\det g|V(K_{1},\dots,K_{n}),\quad g\in\GL(V,\R).$$

We shall use the following extension of mixed volumes. Given $K_1,\dots,K_{n-1}\in\K(V)$, the functional $K \mapsto V(K_1,\ldots,K_{n-1},K)$ can be uniquely extended to a continuous linear functional on the space of continuous $1$-homogeneous functions $f:V^* \to \R$ such that for all $K \in \mathcal{K}(V)$ 
$$V(K_1,\ldots,K_{n-1},h_K)=V(K_1,\ldots,K_{n-1},K).$$

\subsection{McMullen's decomposition}
Let $\Val$ denote the Banach space of real-valued, translation invariant, continuous valuations on $V$.

A valuation $\phi \in \Val$ is called {\it homogeneous of degree $k$} if $\phi(tK)=t^k\phi(K)$ for all $t \geq 0$.  
The subspace of valuations of degree $k$ is denoted by $\Val_k$. 

\begin{Theorem}[McMullen \cite{mcmullen77}]
\begin{equation} \label{mcmullen_dec}
\Val=\bigoplus_{k=0,\ldots,n} \Val_k.
\end{equation} 
\end{Theorem}

Let $Z:\K(V)\to\K(V^*)$ be a continuous, translation invariant Minkowski valuation and $u\in V$ be fixed. Then, McMullen's decomposition implies that
$$h(ZK,u)=\sum_{i=0}^{n}f_{i}(K,u),$$
where $f_{i}(K,u)$ is continuous and satisfies
$$f_{i}(\lambda K,u)=\lambda^{i} f_{i}(K,u),\quad\forall\lambda\in\R_{>0},$$
$$f_{i}(K,\lambda u)=\lambda f_{i}(K,u),\quad\forall\lambda\in\R_{>0}.$$
In \cite{parapatits.wannerer}, it has recently been proved that the functions $f_{i}(K,\cdot)$ are, in general, not support functions. However, in \cite{schneider_schuster06} the following result was proved.

\begin{Lemma}[\cite{schneider_schuster06}]\label{support}Let $V$ be an $n$-dimensional vector space, and $Z:\K(V)\to\K(V^*)$ be a continuous, translation invariant Minkowski valuation. If a convex body $K\in\K(V)$ satisfies
$$h(Z(K),\cdot)=\sum_{i=k}^lf_{i}(K,\cdot),$$
for some $k,l\in\{0,\dots,n\}$, $k\leq l$, then $f_{k}(K,\cdot)$ and $f_{l}(K,\cdot)$ are support functions.  
\end{Lemma}
Moreover, if $Z$ has an invariance property (e.g.\! it is $\SL(V,\R)$-contravariant), then each $f_{i}$ satisfies the same invariance property.

\subsection{Homogeneous real-valued valuations}
In this section, we recall the characterization results on continuous, translation invariant valuations with values in $\R$ we shall need to prove Theorem \ref{contravariant}. For more recent results on real-valued valuations  see, for instance, \cite{alesker_mcmullenconj01,alesker_fourier,alesker_bernig_schuster,klain00,ludwig.reitzner10}.

One of the first characterization results is due to Hadwiger. 
\begin{Theorem}[\cite{hadwiger}]\label{maxDegree}Let $V$ be an $n$-dimensional vector space and let $\phi:\K(V)\to\R$ be a continuous, translation invariant valuation which is homogeneous of degree $n$, i.e. $\phi\in\Val_{n}$. Then $\phi=c\vol_{n}$ with a constant $c\in\R$.
\end{Theorem}

A characterization for valuations of degree $n-1$ was given by McMullen. It will be crucial for the proof of Theorem \ref{contravariant}. 
\begin{Theorem}[\cite{mcmullen80}]\label{mcmullenN1}Let $V$ be an $n$-dimensional vector space and $\phi\in\Val_{n-1}$.
Then there exists a continuous, 1-homogeneous function $\varphi:V^*\to\R$ such that for all $K\in\K(V)$
$$\phi(K)=V(K,\dots,K,\varphi).$$
The function $\varphi$ is unique up to a linear function.
\end{Theorem}

A valuation $\phi\in\Val$ is called \emph{simple} if $\phi(K)=0$ for every $K\in\K(V)$ with $\mathrm{dim} K< n$.
\begin{Theorem}[\cite{klain95, schneider96}]\label{simple}Let $V$ be an $n$-dimensional vector space and  let $\phi:\K(V)\to\R$ be a continuous, translation invariant, simple valuation. Then,
$$\phi(K)=c\vol(K)+V(K,\dots,K,f),$$
where $c\in\R$ is a constant and $f:V^*\to\R$ is an odd, 1-homogeneous, continuous real function unique up to a linear map.
\end{Theorem}

From the previous theorems follow the next two useful results.
\begin{Lemma}[\cite{klain00}]\label{lemma24a}Let $V$ be an $n$-dimensional vector space and $\phi\in\Val_{j}$,
for a given $j\in\{0,1,\dots,n-1\}$.
If $\phi(K)=0$ whenever $\dim K= j$, then $\phi(K)+\phi(-K)=0$.
\end{Lemma}

\begin{Lemma}[\cite{schneider_schuster06}]\label{lemma24b}Let $V$ be an $n$-dimensional vector space and $\phi\in\Val_{j}$,
for a given $j\in\{0,1,\dots,n-1\}$.
If $\phi(K)=0$ whenever $\dim K= j+1$, then $\phi\equiv 0$.
\end{Lemma}

\section{Proof of Theorem \ref{contravariant}}
In this section $W$ denotes a 2-dimensional complex vector space.

Let $Z:\K(W)\to\K(W^*)$ be a continuous, translation invariant Minkowski valuation, which is $\SL(W,\C)$-contravariant. 
Applying McMullen's decomposition \eqref{mcmullen_dec} to $Z$, we get
$$h(ZK,u)=\sum_{i=0}^{4}f_{i}(K,u),$$
where $f_{i}(K,u)$ is a continuous 1-homogeneous function of $u$. Using Lemma \ref{support} we have that $f_{0}(K,\cdot)$ and $f_{4}(K,\cdot)$ are support functions. For a fixed direction $u$, $f_{0}(\cdot,u)$ and $f_{4}(\cdot,u)$ are continuous, translation invariant valuations of degree of homogeneity 0 and 4, resp. Thus, they are a multiple of the Euler characteristic and the volume, resp. (the latter follows from Theorem \ref{maxDegree}), but this is not compatible with the $\SL(W,\C)$-contravariance property unless the multiple is the null function. Therefore, we have 
\begin{equation}\label{hf2}h(ZK,u)=f_{1}(K,u)+f_{2}(K,u)+f_{3}(K,u),\quad\forall u\in W,\,K\in\K(W).\end{equation}
Again by Lemma \ref{support}, $f_{1}$ and $f_{3}$ are support functions of degree of homogeneity 1 and 3, respectively. We claim that $f_{3}(K,u)=h(\Pi_{N}K,u)$ for some $N\in\K(\C)$ and $f_{1}(K,u)=h(\tilde\D_{M}K,u)$, with $M\in\K(\C)$, as given in \eqref{contra1}. Indeed, the Minkowski valuations defined by \eqref{eqcontram} (resp. by \eqref{eqcovm}) are also continuous, translation invariant and $\SL(W,\C)$-contravariant (resp. $\SL(W,\C)$-covariant) Minkowski valuations when $\dim_{\C}W=2$. From the proof of Theorem \ref{contram} (resp. \ref{covm}), no other valuations of fixed degree of homogeneity 3 (resp. 1) can appear even in the 2-dimensional case (see also \cite{abardia12}). Thus, the claim follows directly for the expression of $f_{3}$ and for $f_{1}$, it follows from the identification $\Phi$ between $W$ and $W^*$ induced by the map in \eqref{detId}.

\begin{Lemma}\label{det32}Let $Z:\K(W)\to\K(W^*)$ be a continuous, translation invariant, $\SL(W,\C)$-contravariant Minkowski valuation given by \eqref{hf2}.
Then, $f_{2}:\K(W)\times W\to\R$ is a continuous function satisfying
$$f_{2}(gK,u)=(\mathrm{det}_{\C}(g))^{3/2}f_{2}(K,g^{-1}u)$$ for every $g\in\GL(W,\C)$ with positive determinant. (Here, we denote by $\det_{\C}g$ the determinant of $g$ as a complex endomorphismus of $W$, that is, the determinant of the associated $2\times 2$ complex matrix.)
\end{Lemma}
\begin{proof}
Let $g\in\GL(W,\C)$ have positive determinant. Then, there are $t>0$ and $g_{0}\in\SL(W,\C)$ such that $g=tg_{0}$. Notice that $\det_{\C} g=t^2$. Since $f_{2}$ is 2-homogeneous in the variable of the convex body and 1-homogeneous in the variable of the direction, we have
$$f_{2}(gK,u)=f_{2}(tg_{0}K,u)=t^2f_{2}(K,g_{0}^{-1}u)=t^3f_{2}(K,g^{-1}u)=(\mathrm{det} _{\C}g)^{3/2}f_{2}(K,g^{-1}u),$$
and the result follows.
\end{proof}

\begin{Lemma}\label{LemmaDim2}Let $Z:\K(W)\to\K(W^*)$ be a continuous, translation invariant, $\SL(W,\C)$-contravariant Minkowski valuation which is given by
$$h(ZK,u)=f_{1}(K,u)+f_{2}(K,u)+f_{3}(K,u),$$
where $f_{1}(K,\cdot)=h(\tilde\D_{M}K,\cdot)$ and $f_{3}(K,\cdot)=h(\Pi_{N}K,\cdot)$ for some $M,N\in\K(\C)$, with $\tilde\D_{M}$ and $\Pi_{N}$ defined in \eqref{contra1} and \eqref{contra3}, resp. 
Then, $f_{2}(K,\cdot)\equiv 0$ whenever $\dim K\leq 2.$
\end{Lemma}
\begin{proof}
Let $K\in\K(W)$ be a 2-dimensional convex set. Then, $f_{3}(K,u)=0$ for every $u\in W$ and 
$$h(ZK,u)=f_{1}(K,u)+f_{2}(K,u),$$
so that, by Lemma \ref{support} we have that $f_{2}(K,\cdot)$ is a support function.

Suppose that $K$ is contained in the 2-dimensional real vector space $E=\mathrm{span}_{\R}\{e_{1},e_{2}\}$ with $e_{1}, e_{2}$ linearly independent vectors over $\C$. Define $g\in\GL(W,\C)$ by $ge_{1}=\lambda e_{1}$, $ge_{2}= e_{2}$ with $\lambda\in\R_{>0}$. Notice that $gE=E$. Then, by the previous lemma we have
$$f_{2}(gK,\alpha e_{1})=\lambda^{3/2}f_{2}(K,g^{-1}\alpha e_{1})=\lambda^{1/2} f_{2}(K,\alpha e_{1}),$$
for every $\alpha\in\C$. By Theorem \ref{maxDegree}, 
$$f_{2}(gK,\alpha e_{1})=c(\alpha e_{1})\vol(gK)=\lambda f_{2}(K,\alpha e_{1}).$$
Thus, $$f_{2}(K,\alpha e_{1})=\lambda^{1/2}f_{2}(K,\alpha e_{1}),$$
for every $\lambda>0$, which implies $f_{2}(K,\alpha e_{1})=0$ for every $\alpha\in\C$. In a similar way we get $f_{2}(K,\alpha e_{2})=0$ for every $\alpha\in\C$.
Using that $f_{2}(K,\cdot)$ is the support function of a convex body $Z_{2}K\subset  W^*$, we get that $Z_{2}K\subset (\mathrm{span}\{e_{1},ie_{1}\})^{\circ}$ and $Z_{2}K\subset(\mathrm{span}\{e_{2},ie_{2}\})^{\circ}$, where $F^{\circ}$ denotes the annihilator of the subspace $F\subset W$. Thus, $Z_{2}K=\{0\}$ for every $K\subset E$. 

For $e_{1},e_{2}$ linearly independent over $\C$, the orbit of $E=\mathrm{span}_{\R}\{e_{1},e_{2}\}$ under the action of $\SL(W,\C)$ is dense in the space of 2-dimensional planes in $W$. Since $f_{2}$ is continuous, for every 2-dimensional vector space $E$, we get that $f_{2}(K,\cdot)\equiv 0$ for every $K\in\K(E)$.
\end{proof}

\begin{Lemma}Let $Z:\K(W)\to\K(W^*)$ be as in Lemma \ref{LemmaDim2}.
Then, $f_{2}(K,\cdot)\equiv 0$ whenever $\dim K\leq 3$.
\end{Lemma}
\begin{proof}
By the previous lemma, $K\mapsto f_{2}(K,u)$ is a continuous, translation invariant valuation, homogeneous of degree 2 which vanishes on every 2-dimensional convex body. Thus, Lemma \ref{lemma24a} implies that 
$$f_{2}(K,u)+f_{2}(-K,u)=0,\quad\forall K\in\K(W),u\in W.$$

Let $E\subset W$ be a $3$-dimensional subspace. Then, $E$ can be written as  $E=\mathrm{span}_{\R}\{e_{1},ie_{1},e_{2}\}$ for some vectors $e_{1},e_{2}\in W$ linearly independent over $\C$. For simplicity, we assume that $\{e_{1},ie_{1},e_{2}\}$ constitutes an orthonormal basis of $E$ and we identify $E^*$ with $E$. 

Let $K\in\K(E)$ be a fixed convex body in $E$.
Recall that, $f_{1}(K,\cdot)$ and $f_{3}(K,\cdot)$ are support functions given by $f_{1}(K,\cdot)=h(\tilde \D_{M}K,\cdot)$ and $f_{3}(K,\cdot)=h(\Pi_{N}K,\cdot)$ for some $M,N\in\K(\C)$, with $\tilde\D_{M}$ and $\Pi_{N}$ defined in \eqref{contra1} and \eqref{contra3}, resp. Thus, for every $u\in W$, 
$$f_{3}(K,u)=h(\Pi_{N}K,u)=\int_{S^3}h(N\cdot u,v)dS_{3}(K,v)=\vol_{3}(K)(h(N\cdot u,ie_{2})+h(N\cdot u,-ie_{2})),$$
since $K\in\K(E)$ and $ie_{2}$ is a normal vector to $K\subset W$ (see \cite{gardner_book06,schneider_book93} for more information on the surface area measure of a convex body).
In particular, we obtain 
$$h(\Pi_{N}K,\alpha e_{1}+\beta e_{2})=h(\Pi_{N}K,\beta e_{2}),\quad\forall\alpha,\beta\in\C.$$
Now, as in \cite{schneider_schuster06}, we use the subadditivity of $h(ZK,\cdot)$. We have
\begin{align*}0&\geq h(Z(\lambda K),\alpha e_{1}+\beta e_{2})-h(Z(\lambda K),\alpha e_{1})-h(Z(\lambda K),\beta e_{2})
\\&=h(\Pi_{N}(\lambda K),\alpha e_{1}+\beta e_{2})-h(\Pi_{N}(\lambda K),\alpha e_{1})-h(\Pi_{N}(\lambda K),\beta e_{2})
\\&\quad+f_{2}(\lambda K,\alpha e_{1}+\beta e_{2})-f_{2}(\lambda K,\alpha e_{1})-f_{2}(\lambda K,\beta e_{2})
\\&\quad+h(\tilde \D_{M}(\lambda K),\alpha e_{1}+\beta e_{2})-h(\tilde \D_{M}(\lambda K),\alpha e_{1})-h(\tilde \D_{M}(\lambda K),\beta e_{2})
\\&=\lambda^2\left(f_{2}(K,\alpha e_{1}+\beta e_{2})-f_{2}(K,\beta e_{2})\right)
\\&\quad+\lambda(h(\tilde \D_{M}K,\alpha e_{1}+\beta e_{2})-h(\tilde \D_{M}K,\alpha e_{1})-h(\tilde \D_{M}K,\beta e_{2})).
\end{align*}
Dividing by $\lambda^2$ and taking the limit $\lambda\to\infty$ we obtain 
$$f_{2}(K,\alpha e_{1}+\beta e_{2})\leq f_{2}(K,\beta e_{2}),$$
for every $\alpha,\beta\in\C,\mu\in\R$.

On the other hand, using that $f_{2}(K,-\xi)=-f_{2}(K,\xi)$, it follows that
\begin{equation}\label{linearity}f_{2}(K,\alpha e_{1}+\beta e_{2})=f_{2}(K,\beta e_{2}),\end{equation}
for every $\alpha,\beta\in\C$, $K\in\K(E)$.

Therefore, it remains to prove that $f_{2}(K,\beta e_{2})=0$ for every $\beta\in\C$, $K\in\K(E)$ to conclude that $f_{2}(K,\cdot)=0$ for every $K$ lying in the 3-dimensional subspace $E=\mathrm{span}_{\R}\{e_{1},ie_{1},e_{2}\}$.

\smallskip
Let $u\in W$. By the previous lemma, $K\mapsto f_{2}(K,u)$ restricted to convex bodies in $E$ is a simple, odd valuation (continuous and translation invariant). Using Theorem \ref{simple} we can write 
$$f_{2}(K,u)=V(K,K,\varphi_{u}),\quad u\in W,\,K\in\K(E),$$
where $\varphi_{u}:E^*\cong E\to\R$ is a continuous, 1-homogeneous and odd function, uniquely determined up to a linear function. We will show that $\varphi_{\beta e_{2}}$ is a linear function for every $\beta\in\C$.

We first prove that $\varphi_{\beta e_{2}}$ is linear for fixed $\beta$, when restricted to $\mathrm{span}_{\R}\{e_{1},ie_{1}\}$.
 
Let $\lambda\in\R_{>0}$ and $g\in\GL(W,\C)$ such that $g e_{1}=\lambda e_{1}$, $g e_{2}=e_{2}$. Denote by $\det g|_{E}$ the determinant of the restriction of $g$ to the 3-dimensional vector space $E$. Then, $\det g|_{E}=\lambda^2$ and $\det_{\C}g=\lambda$. Using Lemma \ref{det32}, we get
$$f_{2}(gK,u)=(\mathrm{det}_{\C} g)^{3/2}V(K,K,\varphi_{g^{-1}u})$$
and using the properties of the mixed volumes,
$$f_{2}(gK,u)=(\mathrm{det} g|_{E}) V(K,K,\varphi_{u}\circ g^{-*}),$$
which gives 
$$V(K,K,\varphi_{g^{-1}u})=\lambda^{1/2}V(K,K,\varphi_{u}\circ g^{-*}),$$
from which we can conclude that (see Theorem \ref{mcmullenN1})
\begin{equation}\label{detc}\varphi_{g^{-1}u}=\lambda^{1/2}\varphi_{u}\circ g^{-*}+l_{u,\lambda},
\end{equation}
where $l_{u,\lambda}$ is a linear function depending on $u$ and $\lambda$.

Let $\gamma=\gamma_{1}+i\gamma_{2}$. Taking $u=\beta e_{2}$ in \eqref{detc} and evaluating at $\gamma e_{1}$, we get
$$\varphi_{\beta e_{2}}(\gamma e_{1})
=\lambda^{1/2}\varphi_{\beta e_{2}}(\lambda^{-1}\gamma e_{1})+l_{\beta e_{2},\lambda}(\gamma e_{1})=\lambda^{-1/2}\varphi_{\beta e_{2}}(\gamma e_{1})+l_{\beta e_{2},\lambda}(\gamma e_{1}).$$
Since $l_{\beta e_{2},\gamma}$ is linear, we have that for every $\lambda>0$,
$$\varphi_{\beta e_{2}}(\gamma e_{1})-\varphi_{\beta e_{2}}(\gamma_{1} e_{1})-\varphi_{\beta e_{2}}(\gamma_{2}e_{1})=\lambda^{-1/2}\left(\varphi_{\beta e_{2}}(\gamma e_{1})-\varphi_{\beta e_{2}}(\gamma_{1} e_{1}) -\varphi_{\beta e_{2}}(\gamma_{2}e_{1})\right),$$
and using $\varphi_{\beta e_{2}}$ is a 1-homogeneous function,
\begin{equation}\label{lineare1}
\varphi_{\beta e_{2}}(\gamma e_{1})=\gamma_{1}\varphi_{\beta e_{2}}(e_{1})+\gamma_{2}\varphi_{\beta e_{2}}(ie_{1}).
\end{equation}

\smallskip
In the following we show that $\varphi_{\beta e_{2}}$ is linear on the whole of $E^*$. 
Let $\gamma\in\C$ and $g_{\gamma}\in\SL(W,\C)$ be defined by $g_{\gamma}e_{1}=e_{1}$ and $g_{\gamma}e_{2}=\gamma e_{1}+ e_{2}$. Note that $gE=E$ and $g_{\gamma}^{-1}e_{1}=e_{1}$ and $g_{\gamma}^{-1}e_{2}=-\gamma e_{1}+e_{2}$. Using the $\SL(W,\C)$-contravariance of $f_{2}$ we get
$$f_{2}(g_{\gamma}K,\beta e_{2})=f_{2}(K,g_{\gamma}^{-1}\beta e_{2})=f_{2}(K,-\gamma\beta e_{1}+\beta e_{2}).$$
Hence, \eqref{linearity} implies 
\begin{equation}\label{gbeta}f_{2}(g_{\gamma}K,\beta e_{2})=f_{2}(g_{\gamma'}K,\beta e_{2}),\quad\forall\gamma,\gamma',\beta\in\C,\,K\in\K(E).\end{equation}
Next, we apply the previous identity to the simplex $K=[0,ae_{1},bie_{1},e_{2}]$ with $a,b\in\R\setminus\{0\}$. Note that 
$$g_{\gamma}K=[0,ae_{1},bie_{1},\gamma e_{1}+e_{2}].$$ 
Thus, the surface area measure of $g_{\gamma}K$ is given by 
\begin{align*}S(g_{\gamma}K,\cdot)&=\frac{|ab|}{2}\delta_{-e_{2}}+\frac{|a|\sqrt{1+\gamma_{2}^2}}{2}\delta_{\frac{\sgn(b)}{\sqrt{1+\gamma_{2}^2}}(-ie_{1}+\gamma_{2}e_{2})}+\frac{|b|\sqrt{1+\gamma_{1}^2}}{2}\delta_{\frac{\sgn(a)}{\sqrt{1+\gamma_{1}^2}}(-e_{1}+\gamma_{1}e_{2})}
\\&\quad\quad+\frac{\sqrt{a^2+b^2+(b\gamma_{1}+a(\gamma_{2}-b))^2}}{2}\delta_{\frac{\sgn(a)\sgn(b)}{\sqrt{a^2+b^2+(b\gamma_{1}+a(\gamma_{2}-b))^2}}(be_{1}+aie_{2}-(b\gamma_{1}+a(\gamma_{2}-b))e_{2})}.\end{align*}
This can be easily shown by computing the normal vector and the area of each facet of the simplex.

Next we compute $f_{2}(g_{\gamma}K,\beta e_{2})$ using the function $\varphi_{\beta e_{2}}$ studied above. Letting $\varphi:=\varphi_{\beta e_{2}}$, we get
\begin{align}\nonumber 2f_{2}(g_{\gamma}K,\beta e_{2})&=\sgn(a)\sgn(b)\left(ab\varphi(-e_{2})+a\varphi(-ie_{1}+\gamma_{2}e_{2})\right.
\\&\label{general}\quad\quad\quad\quad\left.+b\varphi(-e_{1}+\gamma_{1}e_{2})+\varphi(be_{1}+aie_{1}-(b\gamma_{1}+a\gamma_{2}-ab)e_{2})\right).\end{align}

We can now show that $\varphi$ restricted to $\mathrm{span}_{\R}\{e_{1},e_{2}\}$ and $\mathrm{span}_{\R}\{ie_{1},e_{2}\}$ is a linear function. 
Choose $\gamma=ib,\gamma'=a$, for which $b\gamma_{1}+a\gamma_{2}-ab=0$. Then, \eqref{gbeta} becomes
$$a\varphi(-ie_{1}+be_{2})+b\varphi(-e_{1})+\varphi(be_{1}+aie_{1})=a\varphi(-ie_{1})+b\varphi(-e_{1}+ae_{2})+\varphi(be_{1}+aie_{1}),$$
which can be written as
\begin{equation}\label{a}
b\varphi(e_{1})=a(\varphi(ie_{1})+\varphi(ba^{-1}e_{1}-be_{2})+\varphi(-ie_{1}+be_{2})),
\end{equation}
or
\begin{equation}\label{b}
a\varphi(ie_{1})=b(\varphi(e_{1})+\varphi(-e_{1}+ae_{2})-\varphi(-ab^{-1}ie_{1}+ae_{2})).
\end{equation}
Taking the limit $a\to\infty$ in \eqref{a} and $b\to\infty$ in \eqref{b}, we get 
$$\varphi(-ie_{1}+be_{2})=-\varphi(ie_{1})+\varphi(be_{2}),\quad\forall b\in\R\setminus\{0\},$$
$$\varphi(-e_{1}+ae_{2})=-\varphi(e_{1})+\varphi(ae_{2}),\quad\forall a\in\R\setminus\{0\}.$$
Using that $\varphi$ is a 1-homogeneous odd function we get the linearity of $\varphi$ restricted to $\mathrm{span}_{\R}\{e_{1},e_{2}\}$ and to $\mathrm{span}_{\R}\{ie_{1},e_{2}\}$, i.e. for every $x,y\in\R$
$$\varphi(xie_{1}+ye_{2})=x\varphi(ie_{1})+y\varphi(e_{2}),$$
$$\varphi(xe_{1}+ye_{2})=x\varphi(e_{1})+y\varphi(e_{2}).$$

It only remains to prove the linearity on the whole of $E$. 
From the above identities, \eqref{general} can be rewritten as
\begin{align*}\nonumber 2f_{2}(g_{\gamma}K,\beta e_{2})&=\sgn(a)\sgn(b)\left(ab\varphi(-e_{2})-a\varphi(ie_{1})+\gamma_{2}a\varphi(e_{2})\right.
\\&\quad\quad\quad\quad\left.-b\varphi(e_{1})+\gamma_{1}b\varphi(e_{2})+\varphi(be_{1}+aie_{1}-(b\gamma_{1}+a\gamma_{2}-ab)e_{2})\right).\end{align*}

We now choose $\gamma=b^{-1}+ib$ and $\gamma'=ib$, so that \eqref{gbeta} implies
$$\varphi(be_{1}+aie_{1}-e_{2})-a\varphi(ie_{1})-b\varphi(e_{1})+(1+ab)\varphi(e_{2})=\varphi(be_{1}+aie_{1})-a\varphi(ie_{1})-b\varphi(e_{1})+ba\varphi(e_{2}).$$
Using \eqref{lineare1} we have
$$\varphi(be_{1}+aie_{1}-e_{2})+\varphi(e_{2})=\varphi(be_{1}+aie_{1})=b\varphi(e_{1})+a\varphi(ie_{1}),$$
that is, for every $a,b\in\R\setminus\{0\}$
$$\varphi(be_{1}+aie_{1}-e_{2})=b\varphi(e_{1})+a\varphi(ie_{1})-\varphi(e_{2}).$$
Since $\varphi$ is $1$-homogeneous and odd, we obtain that $\varphi$ is linear, i.e., for every $x,y,z\in\R$ we have
$$\varphi(xe_{1}+yie_{1}+ze_{2})=x\varphi(e_{1})+y\varphi(ie_{1})+z\varphi(e_{2}).$$

Therefore, $\varphi=\varphi_{\beta e_{2}}:E^*\to\R$ is a linear function for every $\beta\in\C$, which implies that $f_{2}(K,\beta e_{2})=V(K,K,\varphi_{\beta e_{2}})=0$ and, using \eqref{linearity}, we conclude that $f_{2}(K,\cdot)\equiv 0$ for every $K\in\K(E)$. 
\end{proof}

Theorem \ref{contravariant} follows from the previous lemma since for every $u\in W$, we have that $f_{2}(\cdot,u)$ is a continuous, translation invariant valuation which is homogeneous of degree 2 and vanishes whenever $\dim K=3$. Thus, using Lemma \ref{lemma24b} we have that $f_{2}(\cdot,u)=0$, for every $u\in W$, which implies the result. 

\bibliographystyle{plain}
\def\cprime{$'$}

\end{document}